\documentclass[12pt]{article}
\usepackage{amsmath}
\usepackage{amssymb}
\usepackage{amscd}
\usepackage{amsthm}

\theoremstyle{plain}
\newtheorem{Thm}{Theorem}[section]

\theoremstyle{definition}

\newtheorem{Rem}[Thm]{Remark}

\numberwithin{equation}{section}

\title{A product formula for volumes of varieties}
\author{Yujiro Kawamata}

\begin{document}

\maketitle

The volume $v(X)$ of a smooth projective variety $X$ is defined by
\[
v(X) = \text{lim sup} \frac{\dim H^0(X, mK_X)}{m^d/d!}
\]
where $d = \dim X$.
This is a birational invariant.

\begin{Thm}
Let $f: X \to Y$ be a surjective morphism of smooth projective varieties with 
connected fibers.
Assume that both $Y$ and the general fiber $F$ of $f$ 
are varieties of general type.
Then 
\[
\frac{v(X)}{d_X!} \ge \frac{v(Y)}{d_Y!}\frac{v(F)}{d_F!}
\]
where $d_X = \dim X$, $d_Y = \dim Y$ and $d_F=\dim F$.
\end{Thm}

\begin{proof}
Let $H$ be an ample divisor on $Y$.
There exists a positive integer $m_0$ such that $m_0K_Y - H$ is effective.

Let $\epsilon$ be a positive integer.
By Fujita's approximation theorem (\cite{F}), 
after replacing a birational model of $X$,
there exists a positive integer $m_1$ 
and ample divisors $L$ on $F$ such that 
$m_1K_F - L$ is effective and $v(\frac 1{m_1}L) > v(F)- \epsilon$.

By Viehweg's weak positivity theorem (\cite{V}), 
there exists a positive integer $k$ 
such that $S^k(f_*\mathcal{O}_X(m_1K_{X/Y}) \otimes \mathcal{O}_Y(H))$ is 
generically generated by global sections for a positive integer $k$.
$k$ is a function on $H$ and $m_1$.

We have 
\[
\begin{split}
&\text{rank Im}(S^mS^k(f_*\mathcal{O}_X(m_1K_{X/Y})) \to 
f_*\mathcal{O}_X(km_1mK_{X/Y})) \\
&\ge \dim H^0(F,kmL) \\
&\ge (v(F)- 2 \epsilon)\frac{(km_1m)^{d_F}}{d_F!}
\end{split}
\]
for sufficiently large $m$.

Then
\[
\begin{split}
&\dim H^0(X,km_1mK_X) \\
&\ge \dim H^0(Y,k(m_1-m_0)mK_Y) \times 
(v(F)- 2 \epsilon)\frac{(km_1m)^{d_F}}{d_F!} \\
&\ge (v(Y) - \epsilon)\frac{(k(m_1-m_0)m)^{d_Y}}{d_Y!}
(v(F)- 2 \epsilon)\frac{(km_1m)^{d_F}}{d_F!} \\
&\ge (v(Y) - 2 \epsilon)(v(F)- 2 \epsilon)\frac{(km_1m)^{d_X}}{d_Y!d_F!}
\end{split}
\]
if we take $m_1$ large compared with $m_0$ such that 
\[
\frac{(v(Y) - \epsilon)}{(v(Y)- 2 \epsilon)} \ge (\frac{m_1}{m_1-m_0})^{d_Y}.
\]
\end{proof}

\begin{Rem}
If $X = Y \times F$, then we have an equality in the formula.
We expect that the equality implies the isotriviality of the family.
\end{Rem}


Department of Mathematical Sciences, University of Tokyo, 

Komaba, Meguro, Tokyo, 153-8914, Japan 

kawamata@ms.u-tokyo.ac.jp


\begin{thebibliography}{F}
\bibitem{F}
Fujita, Takao.
{\em Approximating Zariski decomposition of big line bundles}. 
Kodai Math. J. 17 (1994), no. 1, 1--3.

\bibitem{V}
Viehweg, Eckart.
{\em Weak positivity and the additivity of the Kodaira dimension 
for certain fibre spaces}. 
Algebraic varieties and analytic varieties (Tokyo, 1981), 329--353, 
Adv. Stud. Pure Math., 1, North-Holland, Amsterdam, 1983. 

\end{thebibliography}
\end{document}